\documentclass[12pt]{amsart}

\usepackage{mathrsfs}
\usepackage{graphicx}
\usepackage{amsmath}
\usepackage{amssymb}
\usepackage{amsthm}
\usepackage{graphicx}

\newtheorem{theorem}{Theorem}
\theoremstyle{plain}

\newtheorem{definition}[theorem]{Definition}
\newtheorem{example}[theorem]{Example}
\newtheorem{proposition}[theorem]{Proposition}

\newcommand{\bR}{\mathbb{R}}
\newcommand{\bRd}{\mathbb{R}^d}

\newcommand{\cC}{\mathcal{C}}
\newcommand{\cS}{\mathcal{S}}
\newcommand{\bE}{\mathbb{E}}
\newcommand{\cJ}{\mathcal{J}}

\newcommand\bel[1]{\begin{equation}\label{#1}}
\newcommand\ee{\end{equation}}

\numberwithin{equation}{section}
\numberwithin{theorem}{section}

\def\ba{\boldsymbol{\alpha}}
\def\bbt{\boldsymbol{\beta}}

\def\Hep{\xi}
\def\Hepo{\mathrm{H}}

\newcommand\opr[1]{\mathrm{#1}}
\newcommand\bld[1]{\boldsymbol{#1}}
\newcommand\zm{\boldsymbol{(0)}}
\newcommand\km{\boldsymbol{\epsilon(k)}}
\newcommand\str[1]{\widetilde{#1}}
\newcommand\mfk[1]{\mathfrak{#1}}
\newcommand\dual[2]{\langle {#1},{#2} \rangle}

\def\Holder{H\"{o}lder }

\begin{document}

\today

\title[Wick Product and Burgers Equation]{Wick Product in The Stochastic Burgers Equation: A Curse or a Cure?}
\author{S. Kaligotla}
\curraddr[S. Kaligotla]{Department of Mathematics, USC\\
Los Angeles, CA 90089 USA\\
tel. (+1) 213 821 1480; fax: (+1) 213 740 2424}
\email[S. Kaligotla]{kaligotl@usc.edu}
\author{S. V. Lototsky}
\curraddr[S. V. Lototsky]{Department of Mathematics, USC\\
Los Angeles, CA 90089 USA\\
tel. (+1) 213 740 2389; fax: (+1) 213 740 2424}
\email[S. V. Lototsky]{lototsky@usc.edu}
\urladdr{http://www-rcf.usc.edu/$\sim$lototsky}


\begin{abstract}
It has been known for a while that a nonlinear equation driven by singular
noise must be interpreted in the re-normalized,  or Wick,
 form. For the stochastic Burgers equation,
 Wick nonlinearity forces the solution to be a generalized process  no matter how regular the random perturbation is, whence the curse.  On the other hand, certain multiplicative random perturbations of the deterministic Burgers equation can only be interpreted in the Wick form, whence the cure. The analysis is based on
 the study of the coefficients of the chaos expansion of the solution at different stochastic scales.
\end{abstract}

\maketitle

\section{Introduction}

Burgers equation,
\begin{equation}
\label{BE-00}
u_t+uu_x=u_{xx},\ t>0,\ x\in \bR,
\end{equation}
first suggested as a simplified model of turbulence
  (Bateman \cite{Burgers-orig0}, Burgers \cite{Burgers-orig, Burgers-orig1})
    is now used to study problems such as  traffic flows
     (Chowdhury et al. \cite{Chowdhury}) and  mass distribution of the large scale structure of the universe (Molchanov et al. \cite{Molchanov}).
 The equation also appears in the study of
 interacting particle systems (Sznitman \cite{Sznitman}).

 Considering random initial condition and/or driving force in  equation \eqref{BE-00} is one way to model and investigate turbulence. The idea
 goes back to Burgers himself and the result is known as the Burgers turbulence.
 A popular option for the driving force is additive space-time white noise $\dot{W}(t,x)$.
 Mathematical theory of the resulting equation,
 \begin{equation}
 \label{BE-S1}
 u_t+uu_x=u_{xx}+\dot{W}(t,x)
 \end{equation}
 has been developed (Bertini et al. \cite{Bertini}). A more general version
 of \eqref{BE-S1} with multiplicative noise,
 \begin{equation}
 \label{BE-S11}
 u_t+uu_x=u_{xx}+f(t,x)+\big(g(u(t,x)\big)_x+h(u(t,x))\dot{W}(t,x),
 \end{equation}
 has also been studied (Gy{\"o}ngy  and Nualart \cite{Gyongy-Burg}).
 It turns out that \eqref{BE-S1} cannot be generalized much further
 while staying within the same mathematical framework, when the
 solution of the equation is a square-integrable random variable with sufficiently
 regular sample paths. In particular, equation
 \begin{equation}
 \label{BE-S2}
 u_t+uu_x=u_{xx}+u_x\dot{W}(t,x)
 \end{equation}
 makes no sense, because the solution, if existed, cannot be regular enough
 to define the point-wise multiplications $uu_x$ and $u_x\dot{W}$.

 Using the Wick product $\diamond$ instead of the usual point-wise multiplication
 makes it possible to study \eqref{BE-S2} and similar equations in the
 framework of white noise theory. The operation $\diamond$ was first
 introduced in quantum field theory (Wick \cite{Wick-orig}). In stochastic
  analysis, the operation is essentially a convolution (Hida and Ikeda \cite{Hida-WickPr-Orig}), and is closely connected with the It\^{o} and
  Skorokhod integrals (Holden et al. \cite[Chapter 2]{HOUZ}). The Wick
  version of \eqref{BE-S2},
  \begin{equation}
  \label{SE-W1}
  u_t+u\diamond u_x=u_{xx}+u_x\diamond \dot{W}(t,x),\ x\in \bR,
 \end{equation}
  has been studied (Benth et al. \cite[Section 5.2, Example 2]{Pot_Nl}), and is known to have a solution in the space of generalized random processes.
  The tools of
  white noise analysis and availability of generalized random processes
  make it possible to consider even more singular noise than $\dot{W}$,
  both additive (Holden et al. \cite{Holden-BurWick-add1, Holden-BurWick-add2})
  and multiplicative (Benth et al. \cite{Pot_Nl}).

  Our objective in this paper is to study the equation
\begin{equation}
\label{general}
\begin{split}
 u_t(t,x)&+u(t,x)\diamond u_x(t,x)=u_{xx}(t,x)+f(t,x)\\
 &+\sum_{k\geq 1} \left(
 a_k(t,x)u_{xx}+b_k(t,x)u_x+c_k(t,x)u+g_k(t,x)\right)\diamond \xi_k,
\end{split}
\end{equation}
 $ \ 0<t\leq T,\ u(0,x)=\varphi(x),$ where  $\{\xi_k,\ k\geq 1\}$ are independent identically distributed (iid)
standard normal random variables, $T<\infty$ is non-random, the coefficients
$a_k,b_k,c_k$ and the free terms $f,g_k$ are non-random,
 and $x\in G$, with
 \begin{itemize}
 \item $G=\bR$ (whole line)
or
\item $G= S^1$ (the circle, which corresponds to periodic boundary conditions).
\end{itemize}
Equation \eqref{general}
includes \eqref{SE-W1} as a particular case because the space-time
white noise $\dot{W}(t,x)$ can be written as
\begin{equation}
\label{STWN}
\dot{W}(t,x)=\sum_{k\geq 1} h_k(t,x)\xi_k,
\end{equation}
where $\{h_k,\ k\geq 1\}$ is an orthonormal basis in
$L_2((0,T)\times\bR)$
(see, for example, Holden et al. \cite[Definition 2.3.9]{HOUZ}).
From the physical point of view,  \eqref{general} is a natural random perturbation
of the original equation \eqref{BE-00}.
Indeed, one possible interpretation
of \eqref{BE-00} is the motion of a one-dimensional fluid, and then
the basic laws of fluid dynamics suggest a more general version of
\eqref{BE-00}:
\begin{equation}
\label{BE-general}
u_t+uu_x=\frac{1}{\rho}(\mu u_x)_x+\frac{1}{\rho}F(t,x),
\end{equation}
where $\rho$ is the density of the fluid, $\mu$ is the (dynamic)
viscosity, and $F$ is the external force (see, for example,
Gallavotti \cite[Section 1.2.2]{Gallavotti}). Thus, \eqref{BE-00} is
\eqref{BE-general} with constant $\mu$ and $\rho$.
If $\mu$ is not known, then (assuming $\rho$ is still constant)
a possible approach  is to consider
\begin{equation}
\label{mu}
\mu(t,x)=\mu_0+\varepsilon \dot{W}(t,x), \ \mu_0=const;
\end{equation}
given the time and space scales of the model, one can  choose $\varepsilon>0$  small enough to have the right-hand side of
\eqref{mu}  positive with probability arbitrarily close to one.
Substituting \eqref{mu} into \eqref{BE-general} and
using \eqref{STWN} (and interpreting all multiplications in the Wick sense)
leads to a particular case of \eqref{general} with
$a_k=\mu_0+h_k$, $b_k=\partial h_k/\partial x$, $f=c_k=g_k=0$. Note that we must interpret multiplications in the Wick
sense: no matter how small the $\varepsilon$ is in \eqref{mu}, the
equation $u_t=(1+\varepsilon \dot{W}(t,x))u_{xx}$ is
ill-posed path-wise.
 An extra benefit of this interpretation is preservation of the mean 
  dynamic, something
we would never get with the usual Burgers equation. Indeed,
a remarkable property of the Wick product is that
 the (generalized) expectation of the product is the product of
expectations:
$\bE(u\diamond v)=(\bE u)(\bE v)$. As a result,
the (generalized) expected value $U=\bE u$ of the solution of \eqref{general}
satisfies the usual Burgers equation
$$
U_t+UU_x=U_{xx}+f,\ U|_{t=0}=\bE \varphi.
$$

The main results of the paper are as follows:
\begin{enumerate}
\item If the functions $a_k, b_k, c_k,f,g_k$ are non-random, bounded and
 measurable and the initial
    condition $\varphi$ is a generalized random field on $G$, then (under some addition technical conditions on $f$ and $\varphi$)
    there exists a unique generalized process solution of \eqref{general};
\item  The Wick product $u\diamond u_x$ forces the  solution of \eqref{general} to be a generalized process even when the
    random perturbation is very benign (such as $a_k=b_k=c_k=g_k=0$ for all
    $k\geq 1$, and $\varphi(x)=\xi_1\,h(x)$ for a smooth compactly supported $h$).
\end{enumerate}
Our approach  to proving existence and uniqueness of solution
is to derive the chaos expansion of the solution of  \eqref{general}
and to establish explicit bounds on the coefficient of the expansion.
The proof has a combinatorial component due to the appearance of the
Catalan numbers. Mikulevicius and Rozovskii \cite{MR-WNSE}
use the same approach to study the Wick-stochastic version of the
Navier-Stokes equations.

The overall conclusion, which also explains the title of the paper, is that,
 while solutions of any stochastic Burgers equation in the Wick form
are generalized random processes (even when the use of the
point-wise multiplication leads to a classical solution --- whence the curse),
certain random perturbations  can only be considered for the
Wick form of the equation (because point-wise multiplication cannot be
defined --- whence the cure). In our opinion, the cures (the ability
to consider more general stochastic perturbations, preservation  of the
mean dynamics, and an easy access to the chaos coefficients of the solution)
outweigh the curses (necessity to work with generalized random elements and
questions about  physical interpretation of the model). \\

Although there are some similarities between our work and that of
Benth et al. \cite{Pot_Nl}, there are also several important differences:
\begin{enumerate}
\item The random perturbations in \eqref{general} and in
\cite{Pot_Nl} are different and there is no obvious way to reduce one
to the other;
\item The solution of \eqref{general} is global in time (can be constructed
for all $T>0$);
\item The solution of \eqref{general} is constructed on an arbitrary
probability space, not just on the white noise space;
\item Our analysis of the chaos expansion of the solution leads to
more detailed information about the solution space and can be used
to compute the solution numerically.
\end{enumerate}
In Section \ref{sec1} we outline the main constructions in the
theory of generalized processes. In Section \ref{sec2} we define the
chaos solution for equation \eqref{general} and state the main
result about existence and uniqueness of solution (Theorem \ref{th:mainRS}).
 We also show (Theorem \ref{th:non-exist})
that the Wick product indeed forces the solution to be a generalized
process. In Section \ref{sec3} we derive the propagator for
equation \eqref{general} (the system of deterministic partial
differential equations describing the propagation of random perturbation
in the equation through different stochastic scales) and
outline the proof of Theorem \ref{th:mainRS}. We also discuss
briefly the general technical issues related to the study of \eqref{general}
 and similar equations.  The details of the
proof of Theorem \ref{th:mainRS} are in Section \ref{sec4}. \\

Throughout the paper, we fix a probability space $\mathbb{F}=(\Omega, \mathcal{F},\mathbb{P})$ and a countable collection of independent and
identically distributed (iid) standard Gaussian random variables
$\bld{\xi}=(\xi_1, \xi_2,\ldots)$ on $\mathbb{F}$. We also assume that
$\mathcal{F}$ is generated by $\bld{\xi}$. By $\bR$ and $\mathbb{C}$
we denote the sets of real and complex numbers, respectively;
$\cC$ and $\cC^1$ denote the spaces of { bounded continuous} and
bounded continuously differentiable functions (with the derivative also
bounded).

\section{Gaussian polynomial chaos and the Wick product}
\label{sec1}

We start with a review of some constructions from the white noise
theory.

\begin{definition}
\label{def:GCS}
 A generalized Gaussian chaos space is collection of the
following four objects $(\mathbb{F},\bld{\xi},H,\opr{Q})$:
\begin{itemize}
\item A probability space $\mathbb{F}=(\Omega, \mathcal{F},\mathbb{P})$;
 \item a collection of iid standard Gaussian random variables
$\bld{\xi}=(\xi_1, \xi_2,\ldots)$ such that
$\mathcal{F}$ is generated by $\bld{\xi}$;
\item A separable Hilbert space $H$;
\item An unbounded, self-adjoint positive-definite operator $\opr{Q}$ on $H$
such that $\opr{Q}$ has a pure point spectrum:
$$
\opr{Q}\mfk{h}_k=q_k\mfk{h}_k,\ k\geq 1,
$$
 the eigenfunctions $\mfk{h}_k$ of
$\opr{Q}$ form an orthonormal basis in $H$, and the eigenvalues
$q_k$ of $\opr{Q}$ satisfy
\begin{equation}
\label{eq:HSh}
\sum_{k\geq 1} \frac{1}{q_k^{\gamma}}<\infty
\end{equation}
for some $\gamma>0$.
\end{itemize}
\end{definition}
Condition \eqref{eq:HSh} ensures that the projective limit of the
domains of $\opr{Q}^n$, $n\geq 1$, is a nuclear space.\footnote{In the
white noise approach, the dual of that nuclear space with the Borel sigma-algebra and a Gaussian measure is the probability space. We are able to use an arbitrary
probability space without any topological structure. For us, the numbers $q_k$ are just weights to ensure convergence of certain infinite sums.}
To simplify some of the future computations, we will assume that
\begin{equation}
\label{eq:eig}
1<q_1\leq q_2\leq \ldots \ \ {\rm and}\ \
\sum_k \frac{1}{q_k} < 1.
\end{equation}
There is no loss of generality involved, as we can always replace the
original operator $\opr{Q}$ with $c\opr{Q}^{\gamma}$, with $c$ sufficiently large. By $\mfk{q}$ we denote the sequence $(q_1,q_2,\ldots)$ of the
eigenvalues of $\opr{Q}$.

\begin{example} {\rm  (1) The traditional white noise constructions
 correspond  to
$H=L_2(\bRd)$ and $\opr{Q}=-\bld{\Delta}+|x|^2+1$. If $d=1$, then
$q_k=2k$.
(2) Stochastic partial differential equations in $\bRd$ correspond
 $H=L_2((0,T)\times \bRd)$,
$\opr{Q}=-{\partial^2}/\partial t^2-\bld{\Delta}+|x|^2+1$,
with periodic boundary conditions in time.
(3) Stochastic partial differential equations in domains or on a manifold
correspond to
$H=L_2((0,T)\times G)$, where $G$ is a smooth
bounded domain or a smooth compact manifold in $\bRd$, and
$\opr{Q}=-{\partial^2}/\partial t^2-\bld{\Delta}$, with
periodic boundary conditions in time and appropriate boundary conditions for
the  Laplace operator $\bld{\Delta}$ on $G$. }
\end{example}

Next, we review some of the
notations related to multi-indices.

A {\em multi-index} $\ba$ is
a sequence $\ba=(\alpha_1, \alpha_2, \ldots)$ of non-negative integers, such that
only finitely many of $\alpha_k$ are different from $0$. The collection of all
multi-indices is denoted by $\cJ$. By definition,
\begin{itemize}
\item  $\bld{\beta}\leq \ba$ if $\beta_k\leq \alpha_k$ for all $k$;
\item $\bld{\beta}< \ba$ if $\beta_k\leq \alpha_k$ for all $k$ and
$\beta_k< \alpha_k$ for at least one $k$.
\end{itemize}
If $\bbt\leq \ba$, then $\ba-\bbt$ is the multi-index $(\alpha_k-\beta_k,\ k\geq 1)$.
For $\ba\in \cJ$ define
$$
|\ba|=\sum_k \alpha_k,\ \ba!=\prod_{k} \alpha_k!.
$$
Special multi-indices and the corresponding notations are
(i) $\zm$, the multi-index with all zeros: $|\zm|=0$; (ii) $\km$, the multi-index
 with $1$ at position $k$ and zeroes elsewhere: $|\km|=1$.

For a sequence $\bld{z}=(z_1,z_2,\ldots)$ of complex numbers and
$\ba\in \cJ$ we write
$$
 \bld{z}^{\ba}=\prod_{k}z_k^{\alpha_k},\ \bld{z}^{r\ba}=\big(\bld{z}^{\ba}\big)^r,\ r\in \bR.
$$

Here are some useful technical results.

\begin{proposition}
\label{prop:aux} Let $\mfk{q}=(q_1,q_2,\ldots)$ be the sequence of
eigenvalues of the operator $\opr{Q}$. Under the assumptions \eqref{eq:eig},
\begin{align}
\label{exp-sum}
&\sum_{\ba\in \cJ} \frac{\mfk{q}^{-\ba}}{\ba!} = \exp\left( \sum_{k\geq 1}
\frac{1}{q_k}\right),
\\
\label{usual-sum}
&\sum_{\ba\in \cJ} \mfk{q}^{-\ba}=\prod_{k\geq 1} \frac{q_k}{q_k-1},
\\
\label{factorial-ineq}
& |\ba|!\leq \mfk{q}^{\ba} \ba!.
\end{align}
\end{proposition}
\begin{proof}
Define $p_k=1/q_k$.
To establish \eqref{exp-sum} and \eqref{usual-sum}, note that, since
$\lim_{k\to \infty} p_k=0$, we have
$$
\exp\big(\sum_k p_k\big)=\prod_ke^{p_k}=\sum_{\ba\in \cJ} \prod_k \frac{p_k^{\alpha_k}}{\alpha_k!},\ \ \
\prod_{k\geq 1} \frac{1}{1-p_k}=\sum_{\ba\in \cJ}\prod_k p_k^{\alpha_k}.
$$
To establish \eqref{factorial-ineq}, let $n=|\ba|$ and
 use the multinomial formula and \eqref{eq:eig} to find
$$
1\geq \left(\sum_k \frac{1}{q_k}\right)^n = \sum_{\ba\in\cJ;|\ba|=n}
\frac{n!}{\ba!\,\mfk{q}^{-\ba}}.
$$
{\em This concludes the proof of Proposition \ref{prop:aux}.}
\end{proof}

 By a theorem of Cameron and Martin \cite{CM}, every square-integrable function $F(\bld{\xi})$ with values in a (complex) Hilbert space $V$ can be written as
\begin{equation}
 \label{Expansion1}
 F(\bld{\xi})=\sum_{\ba \in \cJ} F_{\ba} \Hep_{\ba},
\end{equation}
where
$$
\xi_{\ba}=\frac{1}{\sqrt{\boldsymbol{\alpha!}}}\prod_k \Hepo_{\alpha_k}(\xi_k),
$$
$$
 \Hepo_n(x)=(-1)^n e^{x^2/2}\frac{d^n}{dx^n} e^{-x^2/2}
$$
and
$$
 F_{\ba}=\bE\big(F(\bld{\xi})\Hep_{\ba}\big).
$$
Then
$$
\bE \|F(\bld{\xi})\|_V^2=\sum_{\ba \in \cJ}\|F_{\ba}\|_V^2.
$$
We denote by $\mathbb{L}(\bld{\xi};V)$ the collection of all
square-integrable $V$-valued functions $F(\bld{\xi})$.

 Next, we  construct spaces of stochastic test functions  and generalized random 
 elements. 
\begin{definition}
\label{def-sp}
 For $\rho\in[0,1]$ and $\ell\geq0$,
 \begin{itemize}
  \item the space $(\mathcal{S})_{\rho,\ell}(V)$ is the collection of $\Phi\in\mathbb{L}_2(\bld{\xi};V)$ such that
   \begin{eqnarray*}
    \|\Phi\|^2_{\rho,\ell;V}
    =\sum_{\boldsymbol{\alpha}\in\mathcal{J}}(\boldsymbol{\alpha}!)^{\rho}
    \mfk{q}^{\ell\boldsymbol{\alpha}}
    \|\Phi_{\boldsymbol{\alpha}}\|_V^2<\infty;
   \end{eqnarray*}
  \item the space $(\mathcal{S})_{-\rho,-\ell}(V)$ is the closure of $\mathbb{L}_2(\bld{\xi};V)$ with respect to the norm
   \begin{eqnarray*}
		\|\Phi\|^2_{-\rho,-\ell;V}
=\sum_{\boldsymbol{\alpha}\in\mathcal{J}}(\boldsymbol{\alpha}!)^{-\rho}
\mfk{q}^{-\ell\boldsymbol{\alpha}}
    \|\Phi_{\boldsymbol{\alpha}}\|_V^2;
   \end{eqnarray*}
  \item the space $(\mathcal{S})_{\rho}(V)$ is the projective limit (intersection endowed with a special topology) of the spaces $(\mathcal{S})_{\rho,\ell}(V)$, as $\ell$ varies over all non-negative integers;
  \item the space $(\mathcal{S})_{-\rho}(V)$ is the inductive limit (union endowed with a special topology) of the spaces $(\mathcal{S})_{-\rho,-\ell}(V)$, as $\ell$ varies over all non-negative integers.
 \end{itemize}
 
 A stochastic test function is an element of $(\mathcal{S})_{\rho}(\mathbb{C})$
 for some $\rho\geq 0$. 
 
 A ($V$-valued) { generalized random element} is an element of $(\cS)_{-1}(V)$.
\end{definition}
\noindent
Thus, every  $\Phi\in (\cS)_{-1}(V)$ has a {\em chaos expansion}
\begin{equation}
\label{ChaosExp}
\Phi=\sum_{\ba}\Phi_{\ba}\xi_{\ba},
\end{equation}
and the coefficients $\Phi_{\ba}$ provide information about $\Phi$
at different stochastic scales (that is, in the linear spans of $\xi_{\ba},\ |\ba|=n$).

In the white noise theory, the spaces $(\cS)_{-0}(\mathbb{C})$ and
$(\cS)_{-1}(\mathbb{C})$ are known,
respectively, as  the spaces of Hida and Kondratiev distributions (note that
indeed $(\cS)_{-0}$ is not the same as $(\cS)_{0}$).

It follows from \eqref{factorial-ineq}
that $\alpha!$ in the definitions of the spaces
can be replaced with $|\alpha|!$: it will not change $(\mathcal{S})_{\rho}$ and
$(\mathcal{S})_{-\rho}$ and will shift the index $\ell$ in the individual $(\mathcal{S})_{\pm\rho,\pm\ell}$. We will see below why the values of
$\rho$ are restricted to $[0,1]$.

For $\Psi \in (\mathcal{S})_{-\rho}(V)$ and $\eta\in(\mathcal{S})_{\rho}(\mathbb{C})$, define $\dual{\Psi}{\eta}\in V$ by
\begin{equation}
\label{dual}
\dual{\Psi}{\eta}=\sum_{\ba\in \cJ}\Psi_{\ba}\eta_{\ba}.
\end{equation}
By the Cauchy-Schwarz and the triangle inequalities,
\begin{equation}
\label{dual1}
\|\dual{\Psi}{\eta}\|_V\leq \left(\sum_{\ba\in \cJ} \|\Psi_{\ba}\|_V^2
(\ba!)^{-\rho}\mfk{q}^{-\ell\ba}\right)^{1/2}
\left(\sum_{\ba\in \cJ}|\eta_{\ba}|^2
(\ba!)^{\rho}\mfk{q}^{\ell\ba}\right)^{1/2}.
\end{equation}

Let $\bld{z}=(z_1,z_2,\ldots)$ be a sequence of complex numbers such that
$\sum_{k\geq 1} |z_k|^2<\infty$.
The {\em stochastic exponential} of $\bld{z}$ is the random variable
\begin{eqnarray*}
\mathcal{E}(\bld{z})=\sum_{\boldsymbol{\alpha}\in \cJ}\frac{\bld{z}^{\boldsymbol{\alpha}}}{\sqrt{\boldsymbol{\alpha}!}}\,
\xi_{\boldsymbol{\alpha}}.
\end{eqnarray*}
Direct computations using \eqref{exp-sum}, \eqref{usual-sum} and
the generating function formula for the Hermite polynomials  show that
\begin{enumerate}
\item If $\sum_{k}q_k^{\ell}|z_k|^2<\infty$, then $\mathcal{E}(\bld{z})\in(\mathcal{S})_{\rho, \ell}(\mathbb{C})$ for every
     $0\leq \rho<1$;
    \item $\sum_{k}q_k^{\ell}|z_k|^2<1$, then
    $\mathcal{E}(\bld{z})\in(\mathcal{S})_{1, \ell}(\mathbb{C})$.
    \end{enumerate}

\begin{definition}
\label{strans}
The $S$-transform $\str{\Phi}$ of
$\Phi\in (\mathcal{S})_{-\rho,-\ell}(V)$ with chaos expansion
 \eqref{ChaosExp}  is
\begin{eqnarray*}
\str{\Phi}(\bld{z})=\langle \Phi, \mathcal{E}(\bld{z}) \rangle=
\sum_{\ba\in \cJ} \frac{\Phi_{\ba}}{\sqrt{\ba!}}\, \bld{z}^{\ba},
\end{eqnarray*}
defined for  $\bld{z}$  such that $\mathcal{E}(\bld{z})\in (\mathcal{S})_{\rho,\ell}(\mathbb{C})$.
\end{definition}

Note that $\Phi_{\bld{(0)}}=\str{\Phi}(0)$;  $\Phi_{\bld{(0)}}$
is called a {\em generalized expectation} of $\Phi$ and
even denoted by $\bE\Phi$, although
it is clear from the corresponding definitions that
a typical  element of $(\cS)_{-\rho,-\ell}$ for $\rho,\ell>0$ cannot
have any moments in the usual sense.

The following is a (partial)
characterization of the spaces $\mathcal{S}_{-\rho}(V)$
in terms of the $S$-transform.
\begin{theorem}[Characterization Theorem]
\label{characterization}

(a) If $\Phi\in(\mathcal{S})_{-\rho,-\ell}$ and $0\leq \rho<1$ then, for every
{\em real} sequences $\bld{p}$ and $\bld{r}$, such that $\mathcal{E}(\bld{p})\in (\mathcal{S})_{\rho,\ell}(\mathbb{C})$, $\mathcal{E}(\bld{r})\in (\mathcal{S})_{\rho,\ell}(\mathbb{C})$,
the function  $f(z)=\widetilde{\Phi}(z\bld{p}+\bld{q})$ is an entire function of $z\in\mathbb{C}$.

(b)
For $0<R,\ell<\infty$ let
\begin{eqnarray}
\label{neighb-zero}
\mathbb{K}_\ell(R)=\left\{ \bld{z}\ :\
\sum_{\boldsymbol{\alpha}\in\mathcal{J}}
\mfk{q}^{\ell\boldsymbol{\alpha}}
    |\bld{z}^{\boldsymbol{\alpha}}|^2<R^2\right\}.
\end{eqnarray}
If $\Phi\in(\mathcal{S})_{-1}(V)$, then there exist $R,\ell$ such that
$\str{\Phi}(\bld{z})$ is analytic in $\mathbb{K}_\ell(R)$.
Conversely, if $f=f(\bld{z})$ is a function analytic in $\mathbb{K}_\ell(R)$
for some $0<R,\ell<\infty$,
then there exists a unique $\Phi\in(\mathcal{S})_{-1}(V)$ such that
$\str{\Phi}=f$.

\end{theorem}
\begin{proof}
(a) See Kuo \cite[Theorem 8.10]{kuo}. There is also a converse statement
characterizing entire functions with certain growth at infinity as $S$-transforms of
elements from $(\mathcal{S})_{-\rho}(V)$.

(b) See Holden et al.  \cite[Theorem 2.6.11]{HOUZ}.
Note that, given $f=\str{\Phi}$, we recover $\Phi=\sum_{\ba} \Phi_{\ba}\xi_{\ba}$
by
\begin{equation}
\label{coef-main}
\Phi_{\ba}=\frac{1}{\sqrt{\ba!}}
\frac{\partial^{|\ba|}\str{\Phi}(\bld{z})}{\partial z_1^{\alpha_1}\partial z_2^{\alpha_2}\cdots}\Big|_{\bld{z}=0}.
\end{equation}
\end{proof}

In what follows, we  refer to  $\mathbb{K}_\ell(R)$ from
\eqref{neighb-zero} as a
 {\em neighborhood} of $\bld{z}=0$.

The significance of Theorem \ref{characterization} is that it provides an
{\em intrinsic} characterization of the spaces $(\cS)_{-\rho}(V)$
in terms of the $S$-transform and the spaces of analytic functions,
as opposed to the {\em extrinsic} characterization of Definition \ref{def-sp}
using a particular orthonormal basis in $\mathbb{L}_2(\bld{\xi};V)$.
We also see why the values of $\rho$ are restricted to $[0,1]$: for
$\rho>1$,   stochastic exponentials are never test functions,  making it 
 impossible to define the $S$-transform.

Finally, we define the Wick product $\diamond$ for two elements
of $(\mathcal{S})_{-1}(\mathbb{C})$.

\begin{definition}
Given $\Phi,\Psi\in (\mathcal{S})_{-1}(\mathbb{C})$,
$\Phi\diamond \Psi$ is the unique element
of $(\mathcal{S})_{-1}(\mathbb{C})$ such that
\begin{equation}
\label{WP1}
\str{\Phi\diamond \Psi}(\bld{z})=\str{\Phi}(\bld{z})\str{\Psi}(\bld{z})
\end{equation}
\end{definition}
An immediate consequence of \eqref{WP1} is a useful property of
generalized expectations:
$$
\bE(\Phi\diamond \Psi)=\big(\bE\Phi\big)\big(\bE \Psi\big).
$$
Formula \eqref{WP1} is motivated by the following property of the
Wick product (which can be traced back to the original paper
by Wick \cite{Wick-orig}):
\begin{equation}
\label{WP2}
\Hepo_m(\xi_k)\diamond \Hepo_n(\xi_{l})=
\begin{cases}
\Hepo_m(\xi_k) \Hepo_n(\xi_{l}),& \ {\rm if} \ k\not=l;\\
\Hepo_{m+n}(\xi_k),& \ {\rm if} \ k=l.
\end{cases}
\end{equation}
Writing $\Hepo_{\ba}(\bld{\xi})=\sqrt{\ba!} \xi_{\ba}$,
we conclude from \eqref{WP2} that
\begin{equation}
\label{WP3}
\Hepo_{\ba}(\bld{\xi})\diamond \Hepo_{\bld{\beta}}(\bld{\xi})=
\Hepo_{\ba+\bld{\beta}}(\bld{\xi}).
\end{equation}
If
$$
\Phi=\sum_{\ba} \bar{\Phi}_{\ba}\Hepo_{\ba}(\bld{\xi}), \ \
\Psi=\sum_{\ba} \bar{\Psi}_{\ba}\Hepo_{\ba}(\bld{\xi}),
$$
then formal term-by-term multiplication using \eqref{WP3}
implies
\begin{equation}
\label{WickProd-main}
\Phi\diamond \Psi=\sum_{\ba,\bld{\beta}} \bar{\Phi}_{\ba}\bar{\Psi}_{\bld{\beta}}
\Hepo_{\ba+\bld{\beta}}(\bld{\xi}),
\end{equation}
which is consistent with \eqref{WP1}. Note also that
if
$$
\Phi=\sum_{\ba} {\Phi}_{\ba}\xi_{\ba} \ \
\Psi=\sum_{\ba} {\Psi}_{\ba}\xi_{\ba},
$$
then $\Psi_{\ba}=\bar{\Phi}_{\ba}\sqrt{\ba!}$ and
\begin{equation}
\label{WickProd-main1}
\Phi\diamond \Psi=\sum_{\ba,\bld{\beta}} \sqrt{\frac{(\ba+\beta)!}{\ba!\bld{\beta}!}}{\Phi}_{\ba}{\Psi}_{\bld{\beta}}
\xi_{\ba+\bld{\beta}},
\end{equation}
which is not as convenient as \eqref{WickProd-main}.

\section{Existence and uniqueness of solution}
\label{sec2}

Let $G=\bR$ or $G=S^1$ (the circle, which corresponds to periodic boundary
conditions and which we identify with $\bR/(0,\pi)$).
Denote by $H^{\gamma}(G)$, $\gamma\in\bR$, the Sobolev spaces on $G$
 and by $\|\cdot\|_{\gamma}$ the corresponding norms;
  $H^0(G)=L_2(G)$. Recall that, for $G=\bR$,
  $$
  \|f\|_{\gamma}^2=\int_{\bR}(1+y^2)^{\gamma}|\hat{f}(y)|^2dy,
  $$
  where $\hat{f}$ is the Fourier transform of $f$, and for $G=S^1$,
  $$
  \|f\|_{\gamma}^2=\sum_k (1+k^2)^{\gamma} |f_k|^2,
  $$
  where $f_k$ are the Fourier coefficients of $f$.  Also recall that,
  by the Sobolev embedding theorem, every element (equivalence class) from $H^1(G)$ has a representative  that is a bounded and \Holder
   continuous function on $G$.
 To avoid unnecessary technical complications, we will simply say
 that if  $f\in H^1(G)$, then $f$ is bounded and \Holder continuous, and
 $\sup_{x\in G}|f(x)|\leq C\|f\|_1$ for some $C$ independent of $f$.

Consider the equation
\begin{equation}
 \label{SB-gen}
 \begin{split}
  u_t(t,x)&+u(t,x)\diamond u_x(t,x)=u_{xx}(t,x)+f(t,x)\\
  &+\sum_{k\geq 1}
  (a_k(t,x)u_{xx}+b_k(t,x)u_x+c_k(t,x)u+g_k(t,x))\diamond \xi_k, \\
  &\ 0<t\leq T,\ x\in G,
 \end{split}
\end{equation}
with initial condition $u(0,x)\in (\cS)_{-1}(L_2(G))$,
where $\xi_k$ are iid standard normal random variables,
and $f,a_k,b_k,c_k$ and $g_k$ are {\em non-random} measurable functions.
To define the solution we assume that $f,g_k\in L_2((0,T)\times G)$ and
\begin{equation}
\label{bound0}
\sup_{t,x}|a_k(t,x)|+\sup_{t,x}|b_k(t,x)|+\sup_{t,x}|c_k(t,x)|+
\|g_k\|_{L_2((0,T)\times G)} \leq q_k^{r}
\end{equation}
for $q_k$ from Definition \ref{def:GCS} and some $r>0$.

\begin{definition}
\label{def-solution}
 The process  $u\in (\cS)_{-1}\big((L_2((0,T);H^2(G))\big)$,
 is called a {\em strong chaos solution} of \eqref{SB-gen} if
 there exist  $0<R,\ell<\infty$ such that, for all
 $\bld{z}=(z_1,z_2,\ldots) \in \mathbb{K}_\ell(R)$ and
 all $t\in [0,T]$,   the equality
 \begin{equation}
\label{stransburger}
 \begin{split}
  \str{u}(t,x;\bld{z})
  &+\int_0^t\str{u}(s,x;\bld{z})\str{u}_x(s,x;\bld{z})ds
  =\str{u}(0,x;\bld{z})+\int_0^t\str{u}_{xx}(s,x;\bld{z})ds\\
  &+\sum_{k\geq 1}
  \int_0^t\Big(a_k(s,x)\str{u}_{xx}+b_k(s,x)\str{u}_x
  +c_k(t,x)\str{u}+g_k(t,x)\Big)z_kds
 \end{split}
\end{equation}
 holds in $L_2(G)$.
 \end{definition}

The a priori assumption $u\in (\cS)_{-1}\big((L_2((0,T);H^2(G))\big)$ implies
$$
\dual{u}{\eta}\in L_2((0,T);H^2(G))
$$
for all $ \eta\in (\cS)_{1}(\mathbb{C})$.
The $S$-transform $\str{u}$ of $u$ is thus an element of
$L_2((0,T);H^2(G))$, so that
$\str{u}_x\in L_2((0,T);H^1(G))$ and
$\str{u}_x\in L_2((0,T)\times G)$.
By the Sobolev embedding theorem, $\str{u}, \str{u}_x \in
L_2((0,T);\cC(G))$, making the product
$\str{u}\str{u}_x$  well defined. Condition
\eqref{bound0} ensures uniform in $(t,x)$ convergence of
all infinite sums.

The chaos solution is a variational solution in the
space $(\cS)_{-1}\big(L_2((0,T);H^2(G))\big)$, and  the characterization theorem
(Theorem \ref{characterization})
makes it possible to restrict the set of test functions to
   stochastic exponentials. The resulting deterministic equation for the $S$-transform
\eqref{stransburger}
(which follows from \eqref{SB-gen} by the linearity of the $S$-transform and
the definition of the Wick product  \eqref{WP1})
can be satisfied in a variety of ways (classical, variational, viscosity, etc.)
Our a priori assumptions on $u$ allow us to satisfy the equation for
$\str{u}$ in the {\em strong}
sense, because all the necessary partial derivatives of $\str{u}$ in $t$ and $x$
exist as $L_2$ functions. This is why  the resulting solution of \eqref{SB-gen}
is called a {\em strong chaos solution}.

While our main objective is to establish existence and uniqueness of
solution of \eqref{SB-gen}, we will start by showing that the space
$(\mathcal{S})_{-1}$ is indeed the natural solution space.
In other words, we should not expect the solution of
\eqref{SB-gen} to belong to any $(\cS)_{-\rho}$ if $\rho<1$.

\begin{theorem}
\label{th:non-exist}
Let $\phi=\phi(x), \ x\in \bR$ be a smooth compactly supported
function and let $\xi$ be a standard Gaussian random variable.
Then equation
\begin{equation}
\label{eq:easy}
u_t+u\diamond u_x=u_{xx},\ 0<t\leq T,\ x\in \bR,
\end{equation}
with initial condition $u(0,x)=\xi\, \phi(x)$ cannot have a solution
that is an element of $(\cS)_{-\rho,-\ell}(L_2((0,T); H^2(\bR))$
for some $0\leq \rho<1$ and $\ell>0$.
\end{theorem}

\begin{proof} We will show that the
$S$-transform $\str{u}$ of $u$ cannot be an entire function, as required
by Theorem \ref{characterization} for $u$ to be an element of $(\cS)_{-\rho,\ell},$
 $\rho<1$.

By \eqref{stransburger},
\begin{equation}
\label{Str-Be1}
\str{u}_t(t,x;z)+\str{u}(t,x;z)\str{u}_{x}(t,x;z)=\str{u}_{xx}(t,x;z),
\end{equation}
$\str{u}(0,x;z)=z\phi(x)$; with only one random variable $\xi$, we have only
one complex parameter $z$.

Equation \eqref{Str-Be1} is the
usual Burgers equation and has a closed-form solution via the
Hopf-Cole transformation (see Evans \cite[Section 4.4.1]{Evans}):
writing $F(x)=\int_{-\infty}^x \phi(y)dy$, we find
\begin{equation}
\label{Str-Be2}
\str{u}(t,x;z)=\frac{\int_{-\infty}^{+\infty}
z\phi(y) \exp\left(-\frac{(x-y)^2}{4t}-\frac{zF(y)}{2}\right)dy}{\int_{-\infty}^{+\infty}
 \exp\left(-\frac{(x-y)^2}{4t}-\frac{zF(y)}{2}\right)dy}.
 \end{equation}
This is a classical solution of \eqref{Str-Be1} and leads to a
classical chaos solution of \eqref{eq:easy}; the solution of \eqref{Str-Be1}
is unique in the
 class $L_2((0,T);H^2(\bR))$ (see Biler et al.
 \cite[Theorem 2.1]{fractal-R}).
 If $u(t,x)\in (\mathcal{S})_{-\rho,-\ell}(\bR)$ for every $t,x$, then, by Theorem \ref{characterization}(a), $\str{u}(t,x;z)$ must be an analytic function of
$z$ for all $t\in [0,T]$, $x\in \bR$, $z\in \mathbb{C}$.
 Since it  is not immediately clear from \eqref{Str-Be2} whether the
 dependence of $\str{u}$ on $z$ is analytic, we will transform
 equation \eqref{Str-Be2} further by going back to the derivation of the
 Hopf-Cole transformation.

 Consider the  equation
 \begin{equation}
 \label{Str-Be3}
 v_t+\frac{1}{2}|v_x|^2=v_{xx},\ 0<t\leq T,\ x\in \bR,
 \end{equation}
 with initial condition $v(0,x;z)=zF(x)=z\int_{-\infty}^x \phi(y)dy$.
 Information about $v$ leads to information about $\str{u}$
  because of the relation $v_x=\str{u}$.
 Note that $F$ is a smooth bounded function.
 Direct computations (see \cite[Section 4.4.1]{Evans})
 show that the function $V$ defined by
 \begin{equation}
 \label{Str-Be4}
 V(t,x;z)=\exp\Big(-v(t,x;z)/2\Big)
 \end{equation}
 satisfies the heat equation 
 \begin{equation}
 \label{Str-Be5}
 V_t=V_{xx},\ V(0,x;z)=\exp\Big(-zF(x)/2\Big).
 \end{equation}
 Equation \eqref{Str-Be5} has a classical solution which is unique in
 the class of bounded twice continuously differentiable functions. Interpreting the
 heat kernel as the normal density, the solution of \eqref{Str-Be5}
 can be written as
 \begin{equation}
 \label{Str-Be6}
 V(t,x;z)=\bE e^{-z\zeta(t,x)},
 \end{equation}
 where $\zeta(t,x)=(1/2)F\big(x+\sqrt{2}w(t)\big)$ and $w$
 is a standard Brownian
 motion.

 We are now ready to show that the function $\str{u}$ from
 \eqref{Str-Be2} cannot be entire (that is, analytic for all $z\in \mathbb{C}$).
   The argument goes as follows:
 \begin{enumerate}
 \item The function $V$ defined in \eqref{Str-Be6} is an extension to the
 complex plain of the characteristic function of a uniformly bounded non-degenerate
 random variable $\zeta$ and is therefore an entire function of  the form
  \begin{equation}
  \label{Str-Be61}
  V(t,x;z)=V_0(t,x;z)e^{g(t,x)z},
  \end{equation}
  where $V_0(t,x;z)$ is an entire function with infinitely many
  zeroes (Lukacs \cite[Theorem 7.2.3]{lukacs}).
 \item Representation \eqref{Str-Be4} then implies that
 the function $v$ cannot be an entire function of $z$; otherwise
 $V$ would have no zeroes;
 \item By uniqueness of solutions of \eqref{Str-Be1} and
 \eqref{Str-Be3} we conclude that $v_x(t,x;z)=\str{u}(t,x;z)$, because
 both functions satisfy the same equation with the same initial condition.
 \item By the fundamental theorem of calculus, $v(t,x;z)=v(t,0;z)
     +\int_0^x\str{u}(t,y;z)dy$, which means that if $\str{u}$ were an
     entire function of $z$, so would be $v(t,x;z)-v(t,0;z)$.
 \item If $v(t,x;z)-v(t,0;z)$ is entire, then, by \eqref{Str-Be4},
 so is
 $$
 \bar{V}(t,x;z)=\frac{V(t,x;z)}{V(t,0;z)}
  $$
  and moreover, $\bar{V}$, as a function of $z$, has no zeros
  (because $1/\bar{V}$ corresponds to $v(t,0;z)-v(t,x;z)$ and must be entire
  as well); then \eqref{Str-Be61} and some basic complex analysis imply
 \begin{equation}
 \label{Str-Be62}
 \bar{V}(t,x;z)=h(t,x)e^{zf(t,x)}\ {\rm or}\  V(t,x;z)=V(t,0,z)h(t,x)e^{zf(t,x)}.
 \end{equation}
 \item Finally, \eqref{Str-Be62} and the equality $\str{u}=-2V_x/V$ imply
 that (if we assume  $\str{u}$ is entire, forcing  $v(t,x;z)-v(t,0;z)$ to
 be entire, in turn forcing  $V$ to have the form \eqref{Str-Be62})
  the solution $\str{u}$ of
 \eqref{Str-Be1} must be a linear function of $z$,
  which is impossible and provides the
 required contradiction.
 \end{enumerate}

 Note that the key step in the argument, namely,  showing that the function
 $V(t,x;z)$ cannot have the same zeros in $z$ for all $x$,  can be carried out in
 several different ways.

 \noindent {\em This concludes the proof of Theorem \ref{th:non-exist}}.
 \end{proof}

 Note that equation \eqref{eq:easy} without the Wick product,
 \begin{equation}
 \label{Str-Be7}
 U_t+UU_x=U_{xx},\ U(0,x)=\xi\,\phi(x),
 \end{equation}
has a classical solution which is a very regular stochastic process.
Indeed, the solution is given by the right-hand side of \eqref{Str-Be2}
with $\xi$ instead of $z$, from which it immediately follows that
$U(t,x)$ is a smooth function of $(t,x)$ and
$$
|U(t,x)|\leq c_1|\xi|e^{c_2|\xi|}
$$
for some positive numbers $c_1, c_2$. That is, as a random variable,
$U$ has moments of all orders. By contrast, the solution of
\eqref{eq:easy} (provided it exists) is a generalized random element from
$(\cS)_{-1}(\bR)$. In other words, Theorem  \eqref{th:non-exist}
delivers the ``curse'' part of paper's title: Wick product in the Burgers equation
indeed forces the solution into the largest space of generalized random elements.
The following observation provides a bit of relief: with all the nice properties
of $U$, there is no easy way to get $\bE U(t,x)$; on the other hand, for the solution of  \eqref{eq:easy}, the generalized expectation $\str{u}(t,x;0)$
 solves the deterministic
Burgers equation with zero initial condition and is therefore equal to zero.

\vskip 0.05 in

The next  theorem delivers the ``cure'' part by establishing
solvability of \eqref{SB-gen} in $(\cS)_{-1}\big(L_2((0,T);H^2(G))\big)$
For technical reasons, the conditions  are slightly different for
$G=\bR$ and $G=S^1$; in particular, we have to  consider the ``homogeneous''
case $f=0$ when $G=S^1$. We discuss these and other questions in the following section. Recall that, for a generalized random element $\Phi$,
$\bE\Phi$ denotes its generalized expectation, and $\bE\Phi=\Phi_{\bld{(0)}}=
\str{\Phi}(0)$.

\begin{theorem}
\label{th:mainRS}
Assume that
\begin{itemize}
\item condition \eqref{bound0} holds;
\item $u_0\in (\cS)_{-1,-\ell}(H^1(G))$ for some $\ell>0$ and
$\bE u_0\in H^2(G)$;
\item If $G=\bR$, then also $\bE u_0\in  L_1(\bR)$ and
$f(t,x)=F_x(t,x)$ for a function $F$ such that both  $F$ and $F_x$ are  bounded and \Holder continuous in $(t,x)$;
\item If $G=S^1$, then also $f=0$.
\end{itemize}
Under these assumptions,  equation \eqref{SB-gen} has  a unique solution and
there exists an $n>0$ such that
\begin{equation}
\label{bound-R}
\|u\|_{-1,-n;L_2((0,T);H^2(\bR))}^2+\sup_{0\leq t \leq T}
\|u(t,\cdot)\|_{-1,-n;H^1(\bR)}^2<\infty.
\end{equation}

\end{theorem}

It turns out  that  the generalized expectation $u_{\bld{(0)}}$
of the solution  has special significance. 
In particular, unlike the other  chaos coefficients,
$u_{\bld{(0)}}$ and its partial derivative in $x$ have to be  bounded continuous functions, and the conditions of the theorem ensure that.
This is yet another technical issue which we will discuss later.

\begin{example} {\rm
\label{Ex-main}
Let us apply Theorem \ref{th:mainRS}
 to the Burgers equation with random viscosity
\begin{equation}
\label{eq:MainEx}
u_t+u\diamond u_x=\Big(\big(\mu_0+\dot{W}(t,x)\big)u_x\Big)_x,
\end{equation}
which was mentioned in the introduction as one of the motivations for
investigating \eqref{SB-gen}. To concentrate on the effects of the noise
in the viscosity,
we assume that the initial condition is a non-random smooth function with
compact support.

For the space-time white noise $\dot{W}(t,x)$, we have 
representation
\begin{equation}
\label{eq:spwn-main}
\dot{W}(t,x)=\sum_{k\geq 1} h_k(t,x)\xi_k,
\end{equation}
where $\{h_k,\ k\geq 1\}$ is an orthonormal basis in
$L_2((0,T)\times G)$
(see, for example, Holden et al. \cite[Definition 2.3.9]{HOUZ}).
Then \eqref{eq:MainEx} becomes a particular case of
\eqref{SB-gen} with $a_k=\mu_0+h_k$, $b_k=\partial h_k/\partial x$,
$c_k=g_k=f=0$.

If $G=\bR$, then we take
$$
h_k(t,x)=h_{k_1,k_2}(t,x)=m_{k_1}(t)w_{k_2}(x),\ k_1,k_2\geq 1,
$$
where $m_l$ are normalized sines and cosines, and
$w_l$ are Hermite functions (the standard reference for Hermite
functions is Hille and Phillips \cite[Section 21.3]{Hille-Phil}).
In particular, it is known that $\sup_{x}|w_l(x)|\leq cl^{-1/12}$ and
therefore $\sup_x|w'_l(x)|\leq cl^{5/12}$
(recall that $-w_l''+(1+|x|^2)w_l=2lw_l$).
Then \eqref{bound0} holds with $q_l=2l,\ r=1$.

If $G=S^1$, then we take
$$
h_k(t,x)=h_{k_1,k_2}(t,x)=m_{k_1}(t)m_{k_2}(x),\ k_1,k_2\geq 1,
$$
where $m_l$ are normalized sines and cosines.
In this case $\sup_x|m'_l(x)|\leq cl$, and again
 \eqref{bound0} holds with $q_l=2l,\ r=1$.

As a result, we can apply Theorem \ref{th:mainRS}  with $q_l=2l$ in
both cases, which is the standard choice in the traditional white noise setting.

We can  generalize the model further and introduce, in addition to
 the white noise viscosity, a white noise
random forcing and a white-in-space initial condition, with arbitrary
correlation between the three (by suitably parsing the sequence
$\bld{\xi}$ for the construction of the corresponding random term).

The eigenvalues  of the type $q_l=(2l)^r$, $r\geq 1$ work in most models.
The option to select other $q_l$ can help, for example,
 to deal with exotic stochastic
forcing terms such as  $\sum_k\big(e^{kt}\sin x\big)\xi_k $.

\noindent {\em This concludes Example \ref{Ex-main}.}}
\end{example}

\section{Outline of the proof and further directions}
\label{sec3}

Recall that our objective is to study the equation
\begin{equation}
 \label{SB-gen-d}
 \begin{split}
  u_t(t,x)&+u(t,x)\diamond u_x(t,x)=u_{xx}(t,x)+f(t,x)\\
  &+\sum_{k\geq 1}
  (a_k(t,x)u_{xx}+b_k(t,x)u_x+c_k(t,x)u+g_k(t,x))\diamond \xi_k, \\
  &\ 0<t\leq T,\ x\in G.
 \end{split}
\end{equation}
By definition, the solution is a generalized random element with 
$S$-transform $\str{u}$ satisfying
\begin{equation}
\label{stransburger-d}
 \begin{split}
  \str{u}(t,x;\bld{z})
  &+\int_0^t\str{u}(s,x;\bld{z})\str{u}_x(s,x;\bld{z})ds
  =\str{u}(0,x;\bld{z})+\int_0^t\str{u}_{xx}(s,x;\bld{z})ds\\
  &+\sum_{k\geq 1}
  \int_0^t\Big(a_k(s,x)\str{u}_{xx}+b_k(s,x)\str{u}_x
  +c_k(t,x)\str{u}+g_k(t,x)\Big)z_kds.
 \end{split}
\end{equation}
The most natural approach to proving existence and uniqueness of
solution for \eqref{SB-gen-d}  is to prove existence and
uniqueness of solution for \eqref{stransburger-d}.
This approach is used, for example, by Benth et al. \cite{Pot_Nl}.
Here are some of the reasons why we choose a different approach:
\begin{itemize}
\item The complex-valued Burgers equation \eqref{stransburger-d}
is not easily handled by the Hilbert-space methods
(a major reason is that, for complex numbers, $z^2\not=|z|^2$),
whereas our goal is to get solvability in the Sobolev spaces $H^{\gamma}$.
\item The \Holder spaces could be a natural choice for \eqref{stransburger-d},
but the available  methods lead either to existence that is local in time
or require additional assumptions about the initial conditions
 (in the form of smallness of certain norms). We would like to avoid either
 of these restrictions.
 \item The problem of uniqueness of solution for \eqref{stransburger-d} is
 non-trivial.
 \item Even if we succeed in finding $\str{u}$ and proving that it is unique,
  there is still no easy way to
 get any regularity information about the solution, such as
 estimate  \eqref{bound-R}.
 \end{itemize}
 Our approach is to get information about the solution from
 the chaos expansion. In what follows, we will use the notation
\begin{equation}
\label{NewBasis}
\Hepo_{\ba}(\bld{\xi})=\sqrt{\ba!} \,\xi_{\ba}
\end{equation}
and an equivalent chaos expansion
\begin{equation}
\label{NewBasis1}
u=\sum_{\ba\in \cJ} u_{\ba} \Hepo_{\ba}(\bld{\xi}).
\end{equation}
By \eqref{WickProd-main}, representation \eqref{NewBasis1}
is more convenient for the computation of the Wick product.

Our first step is to derive equations for the coefficients $u_{\ba}$.
As a bonus, we also get an alternative way to prove uniqueness of solution.

By \eqref{coef-main} (and keeping track of the factorial terms), we get
 the following relation between the
chaos coefficients $u_{\ba}$ from \eqref{NewBasis1} and
the $S$-transform $\str{u}$:
\begin{eqnarray}
\label{chaoscoeff}
\str{u}(t,x,\bld{z})=\sum_{\ba\in \cJ} u_{\ba}(t,x)
{\bld{z}^{\ba}}.
\end{eqnarray}
Before we proceed, let us introduce several notations.
 To reduce the number of subscripts, it is convenient to
 re-name the initial condition:  $u_0(x)=\varphi(x)$, and to use alternative
 notations for derivatives: $\dot{u}=u_t$, and $Du=u_x$.
 With the generalized expectation $u_{\bld{(0)}}$ having special significance
 (which also means frequent appearance in complicated formulas),
 we re-name it as well: $u_{\bld{(0)}}=\mfk{u}$. Now we are ready to derive the
 equations for $u_{\ba}$.

\begin{theorem}
\label{mainlemma}
The function $\str{u}$ is a (strong) solution of \eqref{stransburger-d}
if and only if the chaos coefficients $\{u_{\ba},\, \ba\in\cJ\}$
satisfy (also in the strong sense)
 the system of equations
\begin{align}
 \label{zero-p}
 \dot{\mfk{u}}&+\mfk{u}D\mfk{u}=D^2\mfk{u},\ \mfk{u}(0,x)=\varphi_{\zm}(x);\\
 \label{km-order}
 \dot{u}_{\km}&=D\big(Du_{\km}-\mfk{u}u_{\km}\big)
 +a_kD^2\mfk{u}+b_kD\mfk{u}+
 c_k\mfk{u}+g_k,\\
 \notag  u_{\km}(0,x)&=\varphi_{\km}(x);\\
 \label{general-p}
 \dot{u}_{\ba}&=D\big(Du_{\ba}-\mfk{u}u_{\ba}\big)-\sum_{\zm<\bbt<\ba}
 u_{\bbt}Du_{\ba-\bbt}\\
 \nonumber
 &+\sum_{k: \alpha_k>0}
  \big(a_kD^2u_{\ba-\km}+b_kDu_{\ba-\km}+
 c_ku_{\ba-\km}\big),\\
 \notag u_{\ba}(0,x)&=\varphi_{\ba}(x),
 \ |\ba|>1,
\end{align}
and there exists a $p>0$ such that, for all $\ba\in \cJ$,
\begin{equation}
\label{bound-sol-al}
\|u_{\ba}\|^2_{L_2((0,T);H^2(G))}\leq \mfk{q}^{p\ba}.
\end{equation}
\end{theorem}

\begin{proof}
Note that \eqref{bound-sol-al} is equivalent to the condition
$u\in (\cS)_{-1}(L_2((0,T);H^2(G)))$.

Assume that
$\str{u}$ satisfies \eqref{stransburger-d} and is an analytic function of
$\bld{z}$ so that the series \eqref{chaoscoeff} converges
uniformly in $(t,x,\bld{z})$ in some neighborhood of $\bld{z}=0$.
We then substitute \eqref{chaoscoeff} into \eqref{stransburger-d},
compare the coefficients of  $\bld{z}^{\ba}$ for each $\ba$, and
get  \eqref{zero-p}--\eqref{general-p}. Uniform convergence
ensures that all manipulations are legitimate.

For the proof in the opposite direction, we reverse the above argument.
If \eqref{bound-sol-al}
holds, then the series \eqref{chaoscoeff} converges
uniformly in $(t,x,\bld{z})$ in some neighborhood of $\bld{z}=0$.
We then combine equalities
\eqref{zero-p}--\eqref{general-p} into  power series and
get \eqref{stransburger-d}.

\noindent{\em This concludes the proof of Theorem \ref{mainlemma}.}
\end{proof}

The  system of PDEs \eqref{zero-p}--\eqref{general-p}
is called the {\tt propagator} for
 equation (\ref{SB-gen-d}) and describes how the
stochastic equation propagates chaos through difference stochastic scales.
 Note that this
propagation of chaos has no connection with the similar term used for
deterministic equations in
the study of particle systems (e.g. Sznitman \cite{Sznitman}).

Before we proceed, let us note that:
\begin{enumerate}
\item Since $\lim_{k\to \infty} q_k=+\infty$, an upper bound
$\mfk{q}^{p\ba}$ is equivalent to
$C^{|\ba|}\mfk{q}^{p\ba}$: in \eqref{bound-sol-al}, we can always
take a larger $p$, or, in the original construction of the chaos space,  we can switch from $q_k$ to $Cq_k$.
\item The propagator is a lower triangular system and can be solved by
induction on $|\ba|$. Only the first equation, describing
$\mfk{u}=u_{\bld{(0)}}$ is non-linear: it is a deterministic Burgers equation.
All other equations are linear, but with variable coefficients in the lower-order
derivatives; these coefficients are determined by $\mfk{u}$.
\item By Theorem \ref{mainlemma}, uniqueness of solution for the propagator implies uniqueness of solution of (\ref{SB-gen-d}) in
the sense of Definition \ref{def-solution}.
\end{enumerate}

 To prove Theorem \ref{th:mainRS},  it is enough to
 show that the propagator has a unique solution and
 \begin{equation}
 \label{Coef-Estim-Main}
  \|u_{\ba}\|^2_{L_2((0,T);H^2(G))}+
  \sup_{0\leq t \leq T} \|u_{\ba}(t,\cdot)\|^2_{H^1(G)}\leq C^{|\ba|}\mfk{q}^{p\ba},\ p>0;
 \end{equation}
this is actually stronger than \eqref{bound-R}.

  A key step in the analysis of the propagator is the study of
  the Burgers equation \eqref{zero-p}.
The solution of \eqref{zero-p} must be 
\begin{itemize} 
\item a continuously-differentiable in $x$ function,
to act as a variable coefficient in the subsequent propagator equations, and also 
\item an element of  $L_2((0,T);H^2(G))$, to produce a sufficiently
regular free term in \eqref{km-order}.
\end{itemize} 
On the real line, the necessary solvability results
 can be obtained using the Hopf-Cole
transformation, with sufficiently many regularity assumptions on $\varphi$ and
$f$. On the circle, we have to go with the best result we could find:
Kiselev et al. \cite[Theorem 1.1]{fractal-periodic}, who consider $f=0$. 
 
 We will now present the precise results about solvability of the deterministic 
 Burgers equation. 

Consider the equation
\begin{eqnarray}
 \label{det-burg-eqn}
 U_t(t,x)+U(t,x)U_x(t,x)=U_{xx}(t,x)+F_x(t,x)
\end{eqnarray}
with the initial condition $U(0,x)=U_0(x)$, for $x\in G$ where
$G= \bR$ (real line) or
$G= S^1$ (unit circle or periodic boundary conditions).
To define the solution of the equation, we assume that $U_0$ is a
continuous function and $F\in L_2((0,T)\times G))$.

A {\em weak} solution of \eqref{det-burg-eqn} is a continuous
in $t,x$ function $U$ such that, for all smooth functions $\psi$ with
compact support in $G$, the equality
\begin{equation}
\label{DerBerg-WeakSol}
\begin{split}
\int_GU(t,x)\psi(x)dx&-\frac{1}{2} \int_0^t\int_G U^2(s,x)\psi'(x)dxdt=
\int_G U_0(x)\psi(x)dx\\
&+\int_0^t\int_G U(s,x)\psi''(x)dxdt-\int_0^t\int_G F(s,x)\psi'(x)dxds
\end{split}
\end{equation}
holds for all $t\in [0,T]$.

\begin{theorem}[Solvability of the deterministic Burgers equation]
 \label{det_burgers}
$\ \ {  } \ \ $ \\

 (a) If $U_0\in \cC^1(\bR)\bigcap L_1(\bR)$, $U'_0$ is
 \Holder continuous (any non-zero order will work),  and
 $F$, $F_x$,  are bounded and \Holder continuous in $(t,x)$
 functions, then (\ref{det-burg-eqn}) has
   a weak solution $U=U(t,x)$  such that $U$ is bounded and
  has a continuous bounded
 derivative in $x$. The solution is unique in the class of
 bounded continuous functions.

 (b) If $U_0\in \cC^1(\bR)\bigcap
 H^1(\bR)\bigcap L_1(\bR)$, $U'_0$ is
 \Holder continuous (any non-zero order will work), and 
 $F$ and $F_x$ are
  bounded and \Holder continuous in $(t,x)$ functions,
   then  (\ref{det-burg-eqn}) has
   a weak solution $U=U(t,x)$ such that
 \begin{equation}
 \label{DetBurg-SobReg}
 U\in L_2\big((0,T);H^2(\bR)\big)\bigcap \cC((0,T);H^1(\bR)).
 \end{equation}
  The solution is unique in the
 class of bounded continuous functions.

 (c) If $U_0\in H^2(S^1)$ and $F=0$, then (\ref{det-burg-eqn}) has a weak solution  $U=U(t,x)$   such that $U\in\mathcal{C}\big((0,T);H^2(S^1)\big)$
 (by the Sobolev embedding, this implies continuous differentiability in $x$).
 The solution is unique in the class of bounded continuous function.
\end{theorem}

\begin{proof}
For the case $G=\bR$, existence, uniqueness, and regularity of the solution
are established using the Hopf-Cole transformation: we have
\begin{equation}
\label{HCF}
U=-2V_x/V,
\end{equation}
 where $V$ is the solution of the heat equation
\begin{equation}
\label{eq:Heat-multipl}
V_t=V_{xx}-\frac{1}{2}F(t,x)\, V,\ \ V(0,x)=\exp\Big(-\frac{1}{2}\int_{-\infty}^x U_0(y)dy\Big).
\end{equation}
Thus, the study of equation \eqref{DerBerg-WeakSol} is reduced to
the study of equation \eqref{eq:Heat-multipl}.

 (a) A result from Ladyzhensakaya et al. \cite[Theorem IV.5.1]{LSU}
 implies that, under our assumptions, equation \eqref{eq:Heat-multipl} has
a classical solution and the solution is unique in the class of bounded
 continuous functions. Then a result from
 Rozovskii \cite[Theorem 5.1.1]{stoch-evol-sys} provides the
 following probabilistic representation of the solution, with $w=w(t)$ denoting a standard Brownian motion:
\begin{equation}
\label{eq:Heat-multip1}
V(t,x)=\bE\Big( V_0\big(x+\sqrt{2}w(t)\big)
\exp\big(-\frac{1}{2}\int_0^t F(s,x+\sqrt{2}w(t)-\sqrt{2}w(s))ds\big)\Big);
\end{equation}
Since $F$ is bounded and
$$
e^{-(1/2)\|U_0\|_{L_1(\bR)}}\leq V(x)
\leq e^{(1/2)\|U_0\|_{L_1(\bR)}},
$$
it follows from \eqref{eq:Heat-multip1} that there exist   numbers
$0<p_1<p_2$ such that, for all $t,x$, $p_1\leq V(t,x)\leq p_2$. Therefore
 the functions $U=-2V_x/V$ and $U_x=-2(V_{xx}V-V_x^2)/V^2$
are bounded and continuous. Uniqueness of $U$ follows from the uniqueness of
$V$. The reader can also note that considering the Burgers equation in the form
$U_t+UU_x=U_{xx}/2$ would eliminate a lot of various powers of $2$ in
subsequent computations. \\

(b) Under conditions of this part, equation \eqref{eq:Heat-multipl}
still has a classical solution, but the standard parabolic regularity theorem
in Sobolev spaces cannot be applied to $V$ because $V_0$ is not
integrable, and so $V$ does not belong to any Sobolev space.
On the other hand, \eqref{HCF} implies that \eqref{DetBurg-SobReg} will
follow if we can show that $1/V$ is bounded and has two
continuous and bounded derivatives in $x$ (which we did in part (a)), and that 
\begin{equation}
\label{Heat-SobSpReg}
 V_x\in L_2\big((0,T);H^2(\bR)\big)\bigcap \cC((0,T);H^1(\bR)),
 \end{equation}
 which we show next. 
 
 Denoting $V_x$ by $\bar{V}$, we find from \eqref{eq:Heat-multipl}
 that
 $$
 \bar{V}_t=\bar{V}_{xx}-\frac{1}{2}F\,\bar{V}-\frac{1}{2}VF_x,
 $$
 and $\bar{V}(0,x)=-(1/2)U_0(x)V_0(x)$.
 Assumptions of the theorem and the properties of $V$ imply that
 $\bar{V}(0,\cdot)\in H^1$ and $VF_x\in L_2((0,T)\times \bR)$.
 Then the linear parabolic regularity theorem (see Theorem  \ref{exis-para}
 below) in the normal
 triple $(H^2(\bR), H^1(\bR),L_2(\bR))$ implies \eqref{Heat-SobSpReg}.

 By the Sobolev embedding, the  conditions
 $U_0\in \cC^1(\bR)\bigcap  H^1(\bR)$ plus \Holder continuity of
 $U'_0$ can be replaced with a stronger but shorter condition
 $U_0\in H^2(\bR)$, and this is what happens in the statement of
 Theorem \ref{th:mainRS}. In general,  we do have to add the
 condition $U_0\in L_1(\bR)$ (for example, the function $(1+|x|^2)^{-1/2}$
 is in every $H^{\gamma}(\bR)$, but is not in $L_1(\bR)$.) Finally,
 note that the resulting weak solution $U$, being an element of
 $L_2((0,T);H^2(\bR))$, is in fact strong. \\

(c) See  Kiselev et al. \cite[Theorem 1.1]{fractal-periodic}.\\

 \noindent{\em This completes the proof of Theorem \ref{det_burgers}.}
\end{proof}

We were unable to find satisfactory solvability results
for  the circle when an inhomogeneous term $f(t,x)$ is present,
or  for the Dirichlet or Neumann boundary value problems.
As soon as these results are available,  the corresponding analogue
of  Theorem \ref{th:mainRS} will follow immediately
from Theorem \ref{mainlemma}. In fact, one can go even further and
investigate Wick-stochastic versions of other one-dimensional
equations with quadratic
nonlinearity, such as Camassa-Holm, KdV, KPP, and Kuramoto-Sivashinskii
equations,
{\em provided there is a satisfactory solvability result for the
corresponding deterministic nonlinear equation}.

Extension to higher dimensions presents an additional  challenge, as the
Sobolev space $H^1(\bRd)$, $d\geq 2$,  is no longer embedded into the space of
continuous functions. Nevertheless, Mikulevicius and Rozovskii \cite{MR-WNSE} recently studied the Wick-stochastic versions of the Navier-Stokes equations
using $L_p$ solvability results with $p>d$. Extension of our results
 to a $d$-dimensional
system of Burgers equations seems feasible, especially on $\bRd$, where
the Hopf-Cole transformation reduces the system to a scalar heat equation.

Similar to \cite{MR-WNSE}, one can study other properties of the
chaos solution of (particular cases of) \eqref{SB-gen-d},
such as adaptedness and Markov property.
For a number of reasons, these questions fell outside the scope of
the current paper (in fact, to simplify presentation, we do not even
consider a filtration  and introduce random perturbation simply as a countable
set of iid standard Gaussian random variables).
  \\

We now continue with an outline of the proof of Theorem \ref{th:mainRS}.
 The difference between $G=\bR$ and $G=S^1$ stops
with Theorem \ref{det_burgers}.
After that, analysis of the propagator is all about linear parabolic equations
of the form
$$
u_t=\mathcal{A}u+F,
$$
where $\mathcal{A}$ is a linear partial differential operator. The analysis relies on
 the following result.

\begin{theorem}[Linear Parabolic Regularity]
 \label{exis-para}
 
 Let $(V,H,V')$ be a normal triple of separable Hilbert spaces and
 denote by $[w,v]$, $w\in V',\ v\in V$, the duality between
 $V$ and $V'$ relative to the inner product in $H$.
 Consider
   a collection of bounded linear operators $\mathcal{A}(t)$, $0\leq t\leq T,$
   from  $V$ to $V'$ with the following properties:
   \begin{itemize}
   \item For all $x,y\in V$, the function $t\mapsto [\mathcal{A}(t)x,y]$ is measurable;
       \item There exist $\varepsilon>0$ and $C_0\in \bR$
    such that, for all $t\in [0,T]$ and $v\in V$,
 \begin{eqnarray*}
  [ \mathcal{A}(t)v, v ] + \varepsilon \|v\|_V^2 \leq C_0\|v\|^2_H.
 \end{eqnarray*}
 \end{itemize}

  \noindent Then, for every  $u_0\in H$ and $F\in L_2((0,T);V')$,
  there exists a unique  $u\in L_2((0,T);V)$ such that
  the equality
  $$
  u(t)=u_0+\int_0^t \mathcal{A}(s)u(s)ds+\int_0^t F(s)ds
  $$
  holds in $L_2((0,T);V')$. Moreover, $u\in \cC((0,T);H)$ and
  $$
  \int_0^T \|u(t)\|_V^2dt+\sup_{0<t<T} \|u(t)\|_H^2 \leq
  C(\varepsilon, C_0,T)\big(\|u_0\|_H^2+\int_0^T \|F(s)\|_{V'}^2ds\big).
  $$
\end{theorem}

\begin{proof} See, for example, Rozovskii \cite[Theorem 3.1.2]{stoch-evol-sys}.
\end{proof}

To study the propagator equations \eqref{general-p}, we apply Theorem
\ref{exis-para} with
\begin{itemize}
\item $V=H^2(G)$, $H=H^1(G)$, $V'=L_2(G)$;
\item $\mathcal{A}=\partial^2/\partial x^2-\mfk{u}(t,x)\partial/\partial x-\mfk{u}_x(t,x)$,
where $\mfk{u}$ is the solution of \eqref{zero-p};
\item A complicated but manageable free term
\begin{equation}
\label{propag-RHS}
\begin{split}
F=F_{\ba}&=-\sum_{\zm<\bbt<\ba}
 u_{\bbt}Du_{\ba-\bbt}+\sum_{k: \alpha_k>0}
 \big(a_kD^2u_{\ba-\km}+b_kDu_{\ba-\km}\\+
 &c_ku_{\ba-\km}\big).
\end{split}
\end{equation}
\end{itemize}
By Theorem \ref{det_burgers}, the functions
$\mfk{u}$ and $\mfk{u}_x$ are continuous and bounded, so that
 the operator $\mathcal{A}$ is indeed
bounded from $H^2(G)$ to $L_2(G)$. The terms $u_{\bbt}$
have $\bbt<\ba$ and come from the earlier propagator equations.
In particular, by Theorem \ref{exis-para} and the Sobolev embedding,
 $u_{\bbt}\in L^2((0,T);H^2(G))\bigcap \cC((0,T)\times G)$,
 which means that expressions $u_{\bbt}Du_{\ba-\bbt}$ and
 $a_kD^2u_{\ba-\km}$ are square-integrable, and therefore
 $F\in L_2((0,T)\times G)$.

 The main technical complication is that the number of terms
 on the right-hand side of \eqref{propag-RHS}  grows with
 $|\ba|$, and, to satisfy \eqref{Coef-Estim-Main}, we have to control this
 growth. In the following section, we
  derive a recursive relation between the $L_2$ norms
 of $F_{\ba}$ for different $\ba$ and show that the rate of growth in $|\ba|$
 is controlled by the {\em Catalan numbers} $C_{|\ba|-1}$ (along with the
 admissible factor $\mfk{q}^{p\ba}$).
 The exponential growth rate of the Catalan numbers ($C_k\leq 4^k$)
 allows us to complete the proof.

\section{Analysis of the propagator}
\label{sec4}

In this section we carry out the analysis of the propagator
\eqref{zero-p}--\eqref{general-p} and complete the proof of
 Theorem \ref{th:mainRS}. Our main goal is to establish
 \eqref{Coef-Estim-Main}.

 Equation \eqref{zero-p} is covered by Theorem \ref{det_burgers}:
 under the assumptions on $\varphi_{\bld{(0)}}=\bE u_0$ and $f$,
 we know that $\mfk{u}$ and $\mfk{u}_x$ are  bounded and continuous, and
 $$
 \mfk{u}\in L_2((0,T);H^2(G)).
 $$
 Next, consider \eqref{km-order}.
 If  $\ba=\boldsymbol{\epsilon}_k$, we have
 \begin{eqnarray*}
  \dot{u}_{\boldsymbol{\epsilon}_k}&
  =D\big(Du_{\boldsymbol{\epsilon}_k}-u_{\zm}u_{\boldsymbol{\epsilon}_k}\big)+
  F_{\boldsymbol{\epsilon}_k}(t,x),u_{\boldsymbol{\epsilon}_k}(0,x)
  =\varphi_{\boldsymbol{\epsilon}_k},
 \end{eqnarray*}
 where
 \begin{eqnarray*}
  F_{\boldsymbol{\epsilon}_k}(t,x)=a_kD^2u_{\zm}+b_kDu_{\zm}+
  c_ku_{\zm}+g_k
 \end{eqnarray*}
 By Theorem \ref{det_burgers}
 $F_{\boldsymbol{\epsilon}_k}\in L_2\left((0,T)\times G \right).$
 We then solve for $u_{\boldsymbol{\epsilon}_k}$ in the normal
  triple $(H^2(G),H^1(G), L_2(G))$ using  Theorem \ref{exis-para} and
 deduce \eqref{Coef-Estim-Main} for $|\ba|=1$.

 As was mentioned earlier, induction on $|\ba|$, Sobolev embedding, and
 Theorem \ref{exis-para} imply that $F_{\ba}$, the free term in
 \eqref{general-p} (see also \eqref{propag-RHS}), is an element of
 $L_2((0,T)\times G)$. All we need is a bound on the
 norm of $F_{\ba}$ that is consistent with \eqref{Coef-Estim-Main}.

 By the Sobolev embedding, the $L_2$ norm of $F_{\ba}$ with $|\ba|>1$
 is controlled by the norms of $u_{\bbt}$, $\bld{(0)}<\bbt<\ba$, in the
 space $L_2((0,T);H^2(G))\bigcap\cC((0,T);H^1(G))$.
 Accordingly, for $\ba$ with $|\ba|\geq 1$, we define
\begin{eqnarray*}
L_{\ba}^2=\int_0^T\|u_{\ba}(t,\cdot)\|^2_{2}dt
+\sup_{0<t<T}\|u_{\ba}(t,\cdot)\|_1^2=
\|u_{\ba}\|_{L_2((0,T);H^2(G))}^2+
\|u_{\ba}\|_{\cC((0,T);H^1(G))}^2.
\end{eqnarray*}
The next step is to derive  a recursion for $L_{\ba}$.

By Theorem \ref{exis-para} applied to \eqref{general-p},
 \begin{eqnarray}
 \label{lemmaineq}
  L^2_{\ba}\leq\lambda\left(\|\varphi_{\ba}\|^2_1
  +\|F_{\ba}\|^2_{L_2((0,T)\times G)}\right),
 \end{eqnarray}
 and $\lambda$ (here and below) is a number depending only on $T$ and
 some norms of $\mfk{u}$. The value of $\lambda$ can change from line to
 line. With no loss of generality, we assume that $\lambda> 1$, and, with
  $m=\max(\ell,r)$, put
 $$
 L_{\km}=\lambda q_k^m.
 $$

For $|\ba|>1$,  (\ref{lemmaineq}), \eqref{propag-RHS},
and assumptions of Theorem \ref{th:mainRS} imply
 \begin{eqnarray*}
  L_{\ba}&\leq&\lambda\Big(\mfk{q}^{\ell\ba}
  +\sum_{\zm<\bbt<\ba}\
  \|u_{\bbt}\|_{\mathcal{C}([0,T]\times G)}\
 \|u_{\ba-\bbt}\|_{L_2((0,T);H^1(G))}\\
 &&+
  \sum_{k:\alpha_k>0}
  q_k^r\|u_{\ba-\km}\|_{L_2((0,T);H^2(G))}\Big) \\
  &\leq&
  \lambda\Bigg( \mfk{q}^{\ell\ba}+\sum_{\zm<\bbt<\ba}
  L_{\bbt}L_{\ba-\bbt}+ \sum_{k:\alpha_k>0}
  q_k^r L_{\ba-\km}\Bigg).
 \end{eqnarray*}
Define
\begin{eqnarray*}
\tilde{L}_{\ba}=2\lambda\left(\frac{L_{\ba}}{\mfk{q}^{m\ba}}
+1\right).
\end{eqnarray*}
Then
\begin{eqnarray*}
  \tilde{L}_{\ba}\leq  \sum_{\zm<\bbt<\ba} \tilde{L}_{\bbt}\tilde{L}_{\ba-\bbt}.
\end{eqnarray*}
 Let $\left\{A_{\boldsymbol{\alpha}}:\boldsymbol{\alpha}\in\mathcal{J}\right\}$ be a sequence such that
 $A_{\boldsymbol{\epsilon_k}}=\tilde{L}_{\km}$ and
 \begin{equation}
  \label{equalrec}
  A_{\ba}=\sum_{\zm<\bbt<\ba}A_{\bbt}A_{\ba-\bbt},\ |\ba|>1.
 \end{equation}
  It follows by induction that
  \begin{equation}
  \label{eq:LleqA}
  \tilde{L}_{\ba}\leq A_{\ba}
   \end{equation}
   for all $\ba\in\cJ$. Indeed, \eqref{eq:LleqA} is true by
    definition if $|\ba|=1$. Assume \eqref{eq:LleqA} holds for
 all $|\boldsymbol{\alpha}|<N$. For $|\ba|=N$,
 \begin{eqnarray*}
  \tilde{L}_{\ba}&\leq &\sum_{\zm<\bbt<\ba} \tilde{L}_{\bbt}\tilde{L}_{\ba-\bbt}\\
  &\leq& \sum_{\zm<\bbt<\ba}A_{\bbt}A_{\ba-\bbt}=A_{\ba},
 \end{eqnarray*}
 proving \eqref{eq:LleqA} for all $\ba\in\mathcal{J}$.
 \\

 Our next step is to find a bound on $A_{\ba}$, which we do using
 generating functions. Since multi-indices $\ba$ can have non-zero
 entries in arbitrary positions, the argument requires an additional construction.

 Let $d\geq 1$ be an integer  and let $\boldsymbol{\beta}=(\beta_1,\beta_2,\ldots,\beta_d)$ be
  a $d-$dimensional vectors of nonnegative integers with at least one $\beta_i>0$.
  Let
 $M_{\epsilon_i}\in\bR$, $i=1,2,\ldots,d$, be prescribed positive numbers  and let $M_{\boldsymbol{\beta}}$, $|\boldsymbol{\beta}|>1$, be defined through the recursive equation,
 \begin{eqnarray*}
  M^{(d)}_{\boldsymbol{\beta}}
  =\sum_{\boldsymbol{\gamma}
  +\boldsymbol{\delta}=\boldsymbol{\beta},|\boldsymbol{\gamma}|,
  |\boldsymbol{\delta}|>0}{M^{(d)}_{\boldsymbol{\gamma}}
  M^{(d)}_{\boldsymbol{\delta}}}.
 \end{eqnarray*} Consider,
 \begin{eqnarray*}
  F_d(z_1,z_2,....,z_d)
  =\sum_{\boldsymbol{\beta}}{M^{(d)}_{\boldsymbol{\beta}}
  z_1^{\beta_1}z_2^{\beta_2}\ldots z_d^{\beta_d}}.
 \end{eqnarray*}
 Then
 \begin{eqnarray*}
  F^2_d(z_1,z_2,....,z_d)&=&\sum_{|\boldsymbol{\beta}|>1}
  {\left(\sum_{\boldsymbol{\gamma}
+\boldsymbol{\delta}=\boldsymbol{\beta},|\boldsymbol{\gamma}|,
|\boldsymbol{\delta}|>0}{M^{(d)}_{\boldsymbol{\gamma}}
M^{(d)}_{\boldsymbol{\delta}}}\right)z_1^{\beta_1} z_2^{\beta_2}\ldots z_d^{\beta_d}}\\
 &=&\sum_{|\boldsymbol{\beta}|>1}
 {M^{(d)}_{\boldsymbol{\beta}}z_1^{\beta_1}z_2^{\beta_2}\ldots z_d^{\beta_d}}\\
 &=&F_d(z_1,z_2,....,z_d)-\sum_i{M^{(d)}_{\boldsymbol{\epsilon_i}}z_i},
 \end{eqnarray*}
 where $\boldsymbol{\epsilon_i}$ is the $d$-dimensional vector with $1$ at the $i^{th}$ coordinate and $0$ everywhere else. Solving for $F_d$ we get,
 \begin{eqnarray*}
  F_d(z_1,z_2,....,z_d)
  =\frac{1\pm\sqrt{1-4\sum_i{M^{(d)}_{\boldsymbol{\epsilon_i}}z_i}}}{2}.
 \end{eqnarray*}
 By comparing the coefficients of similar powers we obtain,
 \begin{equation}
  \label{finiterec}
  M^{(d)}_{\boldsymbol{\beta}}
  =\frac{1}{|\boldsymbol{\beta|}}
  \binom{2|\boldsymbol{\beta}|-2} {|\boldsymbol{\beta}|-1}
  \binom{|\boldsymbol{\beta}|}{\boldsymbol{\beta}}
\prod_i {(M^{(d)}_{\boldsymbol{\epsilon_i}})^{\beta_i}};
 \end{equation}
 here $\binom{|\bbt|}{\bbt}=|\bbt|!/\bbt!$.
 Let $\Gamma_d=\left\{\boldsymbol{\alpha}
 =(\alpha_1,\alpha_2,\ldots)\in\mathcal{J}:\alpha_i=0\quad \mbox{for all}\quad i>d\right\}$. If
 $\boldsymbol{\alpha}\in\Gamma_d$ then, by 
   (\ref{equalrec}), $A_{\boldsymbol{\alpha}}$ is uniquely determined by
 $\left\{A_{\boldsymbol{\beta}}:
 \boldsymbol{\beta}<\boldsymbol{\alpha}\right\}$. Also,  if $\boldsymbol{\alpha}\in\Gamma_d$ and
  $\boldsymbol{\beta}<\boldsymbol{\alpha}$, then
 $\bbt \in  \Gamma_d$ as well.
 Set $M^{(d)}_{\boldsymbol{\epsilon_i}}=A_{\boldsymbol{\epsilon_i}}$, for each $\epsilon_i\in\Gamma_d$. Then
 \begin{equation}
 \label{Cat1}
 M^{(d)}_{\boldsymbol{\alpha^d}}=A_{\boldsymbol{\alpha}},
  \end{equation}
  where $\boldsymbol{\alpha^d}=(\alpha_1,\alpha_2,\ldots,\alpha_d)$,
 which follows by induction on $|\boldsymbol{\alpha}|$. Indeed, equality
  \eqref{Cat1} is true if  $\boldsymbol{\alpha}\in\Gamma_d$ such that
 $|\boldsymbol{\alpha}|=1$. Assume \eqref{Cat1} is true for $\boldsymbol{\alpha}\in\Gamma_d$ such that $|\boldsymbol{\alpha}|<N$. For $\boldsymbol{\alpha}\in\Gamma_d$ such that
 $|\boldsymbol{\alpha}|=N$,
 \begin{eqnarray*}
  A_{\boldsymbol{\alpha}}&=&\sum_{\boldsymbol{\gamma}
  +\boldsymbol{\delta}=\boldsymbol{\alpha},
  |\boldsymbol{\gamma}|,|\boldsymbol{\delta}|>0}
  {A_{\boldsymbol{\gamma}}A_{\boldsymbol{\delta}}}\\
  &=&\sum_{\boldsymbol{\gamma^d}
  +\boldsymbol{\delta^d}=\boldsymbol{\alpha^d},
  |\boldsymbol{\gamma^d}|,|\boldsymbol{\delta^d}|>0}
  {M^{(d)}_{\boldsymbol{\gamma^d}}M^{(d)}_{\boldsymbol{\delta^d}}}\\
  &=&M^{(d)}_{\boldsymbol{\alpha^d}},
 \end{eqnarray*}
and the  general equality \eqref{Cat1} follows.

 Given any $\boldsymbol{\alpha}\in\mathcal{J}$ there exists a $d\in\mathbb{N}$ such that $\boldsymbol{\alpha}\in\Gamma_d$. For such a $d$,
 $M^{(d)}_{\boldsymbol{\alpha^d}}=A_{\boldsymbol{\alpha}}$. Using (\ref{finiterec}),
 \begin{equation}
  A_{\boldsymbol{\alpha}}=\frac{1}{|\boldsymbol{\alpha}|}
  {2|\boldsymbol{\alpha}|-2 \choose |\boldsymbol{\alpha}|-1}{|\boldsymbol{\alpha}| \choose
  \boldsymbol{\alpha}}\prod_i {A_{\boldsymbol{\epsilon_i}}^{\alpha_i}}.
 \end{equation}
 We now notice that
 $$
 C_n=\frac{1}{n+1}\binom{2n}{n},\  n\geq 0,
 $$
 is $n$th Catalan number, a very popular combinatorial object
 (see, for example, Stanley \cite{Stanley}).
 Hence, re-tracing our arguments back to $L_{\ba}$, we conclude that
 \begin{eqnarray}
 \label{finalbound}
  L_{\ba}\leq (2\lambda)^{2|\ba|}C_{|\ba|-1}
  \binom{|\boldsymbol{\alpha}|}
  {\boldsymbol{\alpha}}\mfk{q}^{m\ba}.
 \end{eqnarray}
The final two observations are
\begin{itemize}
\item $C_n\leq 4^n$, which follows by the Stirling formula;
\item $\binom{|\boldsymbol{\alpha}|}
  {\boldsymbol{\alpha}}\leq \mfk{q}^{\ba}$, which we proved in
  Proposition \ref{prop:aux}.
  \end{itemize}
  Therefore, $L_{\ba}\leq (8\lambda)^{|\ba|}\mfk{q}^{(m+1)\ba}$,
  which establishes  \eqref{Coef-Estim-Main}
  and completes the proof of Theorem \ref{th:mainRS}.

\vskip 0.2in

\def\cprime{$'$}


\vskip 0.2in

\end{document}